\documentclass{article}

\usepackage{amsmath}
\usepackage{amsthm}
\newtheorem{lemma}{Lemma}
\newtheorem{theorem}{Theorem}
\newtheorem{definition}[theorem]{Definition}

\usepackage{graphicx}
\title{Determining monotonic step-equal sequences of any limited length in the Collatz problem}

\author{Longjiang Li, \\School of Information and Communication Engineering,\\ University of Electronic Science and Technology of China,\\ Chengdu, Sichuan, 611731.\\ Email: longjiangli@uestc.edu.cn}

\begin{document}

\maketitle

\begin{abstract}
This paper proposes a formula expression for the well-known Collatz conjecture (or 3x+1 problem),  which can pinpoint all the growth points in the orbits of the Collatz map for any natural numbers. The Collatz map  $Col: \mathcal{N}+1 \rightarrow \mathcal{N}+1$ on the positive integers is defined as $x_{n+1}=Col(x_n)=(3 x_n +1)/2^{m_n}$ where $x_{n+1}$ is always odd and $m_n$ is the step size required to eliminate any possible even values. The Collatz orbit for any positive integer, $x_1$, is expressed by a sequence, $<x_1$; $x_2\doteq Col(x_1)$;  $\cdots$ $x_{n+1}\doteq Col(x_n)$; $\cdots>$ and $x_n$ is defined as a growth point if $Col(x_n)>x_n$ holds, and we show that every growth point is in a format of ``$4y+3$'' where $y$ is any natural number.  Moreover, we derive that, for any given  positive integer $n$, there always exists a natural number, $x_1$, that starts a monotonic increasing or decreasing Collatz sequence of length $n$ with the same step size.  For any given positive integer $n$, a class of orbits that share the same orbit rhythm of length $n$  can also be determined.
\end{abstract}

\section{Introduction}

The Collatz conjecture (or 3x+1 problem) has been explored for more than 80 years\cite{Aliyev2013}\cite{Ma2019}\cite{Aljassas2019} and  verified to be true for all natural numbers up to $2^{100000}-1$ \cite{Ren2018}, but it still defied any formal proof\cite{Sayama2011}\cite{Tao2019}. Currently, almost all rigorous results are asymptotic or probabilistic. For example, Lagarias and Weiss\cite{Lagarias1992} proposed two stochastic models to mimic the dynamics of the Collatz series. Krasikov and Lagarias \cite{Krasikov2003} derived  bounds using difference inequalities for all sufficiently large $x$. Terras \cite{Terras1976} showed that the minimal element $Col_{min}(x)$ of the Collatz orbit is less than $x$ for almost all $x$. Korec\cite{Korec1994} improved the results to  $Col_{min}(x)< x^\theta$ for  almost all $x$ and any $\theta > log 3/log 4 \approx 0.7924$. Recently, Tao \cite{Tao2019} showed $Col_{min}(x) < \log \log \log \log x$ for almost all $x$. 

This paper proposes a formula expression, which can pinpoint all the growth points in the orbits of the Collatz map for any natural numbers, so that the problem may be more tractable as an equation-solving process. We show that, under some constraints, these equations are easy to solve, which leads to the key results of this paper, as follows.
\begin{itemize}
	\item A monotonic increasing Collatz sequence of any limited length with the same step size can be determined.
	\item A monotonic decreasing Collatz sequence of any limited length with the same step size can be determined.
	\item A class of Collatz orbits that share the orbit rhythm of any limited length can be determined.
\end{itemize}

\section{Notations}
Let denote the natural numbers by $\mathcal{N}:= \{0, 1, 2, \cdots\}$, so that $\mathcal{N}+1 = \{1, 2, \cdots\} $ are the positive integers.  The Collatz problem, also known as the $3x + 1$ problem, generates a series $<x_1, f(x_1), f(f(x_1)), \cdots>$ for an arbitrary positive integer $x_1 \in \mathcal{N}+1$ by defining the Collatz operations as $f(x)= 3 x+1$ when $x$ is odd and $f(x)=x/2$ when $x$ is even. The Collatz conjecture asserts that this series always falls into a 4 $\rightarrow$ 2 $\rightarrow$ 1 cycle regardless of $x_1$.  For the purpose of convenience, we define the Collatz map, $Col: \mathcal{N}+1 \rightarrow \mathcal{N}+1$ as follows.

\begin{definition}[The Collatz map]
\begin{equation} \label{map}
x_{n+1}=Col(x_n)=\frac{3 x_n +1}{2^{m_n}}
\end{equation}
where $n \in \mathcal{N}+1$, $x_{n+1}$ is always odd and $m_n \in \mathcal{N}+1$ is the step size required to eliminate any possible even values for $x_n$. 
\end{definition}

Note that we specify that $Col$ combines two Collatz operations. One is to attempt  $f(x)= 3 x+1$ firstly if $x$ is odd, and then to attempt $f(x)=x/2$ if $x$ is even. Thus, $m_n\ge 1$ always holds for $x_1>1$.

Then, for any positive integer, $x_1$, its Collatz orbit is defined as a sequence, $<x_1$; $x_2\doteq Col(x_1)$;  $\cdots$ $x_{n+1}\doteq Col(x_n)$; $\cdots>$.

\begin{definition}[Growth point]\label{def:growth}
	$x_n$ is defined as a growth point, if $Col(x_n)>x_n$ holds.
\end{definition}

\section{Theorems}
\begin{lemma}[Sufficient and necessary condition] \label{lemma:growth}
   Every growth point is in a format of ``$4y+3$'' where $y$ is any natural number.
\end{lemma}
\begin{proof}
Any positive integer $x_n$ can be expressed as in one of four formats, ``$4y$'', ``$4y+1$'', ``$4y+2$'' and ``$4y+3$''.
\begin{itemize}
	\item ``$4y$'':  $4y$ is even, so $m_n$ for $Col(4y)$  is at least $2$. Thus, $Col(4y)\le y < 4y$.
	\item ``$4y+1$'':  $4y+1$ is odd, so we get a series: $4y+1 \rightarrow 3(4y+1)+1= 12y+4 \rightarrow 3y+1$. Thus, $Col(4y+1)= 3y+1 <4y+1 $.
	\item ``$4y+2$'':  $4y+2$ is even, so we get a series: $4y+2 \rightarrow 2y+1$. Thus, $Col(4y+2)= 2y+1 <4y+2 $.
	\item ``$4y+3$'':  $4y+3$ is odd, so we get a series: $4y+3 \rightarrow 3(4y+3)+1= 12y+10\rightarrow 6y+5 $. Thus, $Col(4y+3)= 6y+5 >4y+3 $.
\end{itemize}
	
\end{proof}	

According to (\ref{map}), for any integer $n\ge 2$, we have
\begin{equation}
   x_{n}=Col(x_{n-1})=Col(Col(...(Col(x_1))...))\doteq Col^{n-1}(x_1)
\end{equation}
\begin{equation} \label{eqn:xn}
=\frac{3^{n-1} x_1 + \sum_{i=0}^{n-2} (3^{n-2-i}\times \prod_{j=0}^{i} 2^{m_j}) }{\prod_{j=0}^{n-1} 2^{m_j}}
\end{equation}
where $Col^{0}(x_1)\doteq x_1$, $2^{m_0}\doteq 1$ and $m_j$ is the step size for $x_j$. 

According to Lemma \ref{lemma:growth}, $x_{n}$  is in a format of ``$4y+3$'', if and only if $x_{n}$ is a growth point. So, we define an equation as follows.

\begin{definition}[Formula expression]
\begin{equation} \label{eqn:formula}
4y_{n}+3= x_n
\end{equation}
\end{definition}

By expanding $x_n$ according to  (\ref{eqn:xn}), we rewrite (\ref{eqn:formula}) as follows.
\begin{multline} \label{eqn:formula2}
y_{n}+1 =\frac{3^{n-1} x_1 + \sum_{i=0}^{n-2} (3^{n-2-i}\times \prod_{j=0}^{i} 2^{m_j})+ \prod_{j=0}^{n-1} 2^{m_j}}{4\times \prod_{j=0}^{n-1} 2^{m_j}}\\
\doteq X(x_1, n, m_1, m_2, \cdots, m_{n-1}) \\
\end{multline}

By assuming $m_1=m_2=\cdots= m_{n-1}=m$, we have
		\begin{equation}
	y_{n}+1=\frac{3^{n-1} x_1 + \sum_{i=0}^{n-2} (3^{n-2-i}\times \prod_{j=0}^{i} 2^{m})+ \prod_{j=0}^{n-1} 2^{m}}{4\times \prod_{j=0}^{n-1} 2^{m}}
	\end{equation}
	\begin{equation}
	=\frac{3^{n-1} x_1  +\frac{(2^m)^{n-1}-3^{n-1}}{2^m-3}  +   (2^m)^{n-1}  }{4\times (2^m)^{n-1} } 
	\end{equation}
\begin{equation}\label{eqn:mono}
= \frac{3^{n-1}}{4\times (2^m)^{n-1} } (x_1- \frac{1}{2^m-3})+ \frac{1}{4(2^m-3)}+\frac{1}{4}
\end{equation}

\begin{theorem}[Monotonic increasing sequences with the same step size] \label{thm:incr}
	Given any positive integer $n$, there always exists a natural number, $x_1$, so that $Col^{i+1}(x_1)>Col^{i}(x_1)$ holds with the same step size for every $i \in [0, n-1]$. 
\end{theorem}
\begin{proof}
	By assuming the step size $m_1=m_2=\cdots= m_{n-1}=m=1$, from (\ref{eqn:mono}), we get
	\begin{equation} \label{eqn:incr}
	y_n+1= \frac{3^{n-1}}{2^{n+1} } (x_1+1)
	\end{equation}
Let $x_1= K \times 2^{n+1}-1$, then $y_n= K \times 3^{n-1} -1$ is an integer for any positive integer $K$. Thus, every integer in the format of $x_1=K \times 2^{n+1}-1$ satisfies the requirements  of the theorem.
\end{proof}

Note that $y_n$ corresponds to a growth point. According to (\ref{eqn:incr}), for any $x_1= K \times 2^{n+1}-1$ where $n>1$,  we get a subsequence of the Collatz orbit from $x_1$ to $x_{n+1}=Col(4y_n+3)=6\times K \times 3^{n-1}-1$ with $m_1=m_2=\cdots= m_{n-1}=m=1$. Fig.~\ref{fig:simincr} shows the digital results calculated by computer  about monotonic increasing sequences with $n=7$ and $K=\{1,2,3\}$, which also validates the theorem.

\begin{figure}
	\centering
	\includegraphics[scale=0.6]{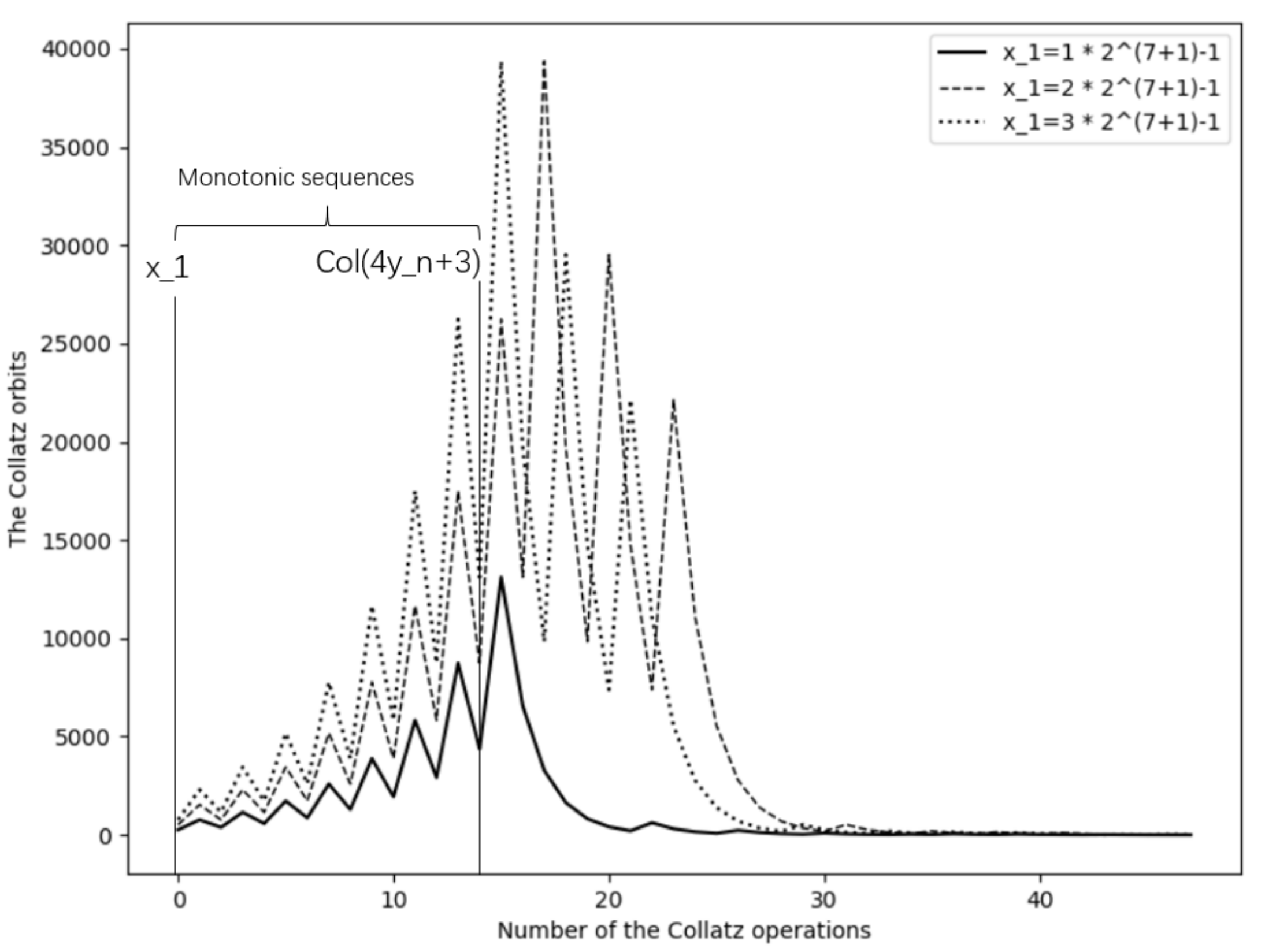}
	\caption{Monotonic increasing sequences with $n=7$ and $K=\{1,2,3\}$ }
	\label{fig:simincr}
\end{figure}

\begin{theorem}[Monotonic decreasing sequences with the same step size]\label{thm:decr}
		Given any positive integer $n$, there always exists a natural number, $x_1$, so that $Col^{i+1}(x_1)<Col^{i}(x_1)$ holds with the same step size for every $i \in [0, n-1]$.
	
\end{theorem}
\begin{proof}
	We can obtain a similar proof just like in Theorem \ref {thm:incr}, by assuming $m_1=m_2=\cdots= m_{n-1}=m_n=m\ge2$, though each sequence is decreasing from $x_1$ to $x_{n+1}$ in this case.
\end{proof}

If two Collatz orbits have the same sequence of step sizes, i.e., $<m_1, m_2, ..., m_n>$, we say that these two orbits share the orbit rhythm of length $n$. 

\begin{theorem}[A class of sequences with the same orbit rhythm]\label{thm:rhythm}
	Given any Collatz orbit with  $<m_1, m_2, ..., m_n \equiv 1>$, there exists a class of orbits that share the same orbit rhythm of length $n$.
\end{theorem}
\begin{proof}
 From 	(\ref{eqn:formula2}), we have
 \begin{equation}
   y^{(2)}_n-y^{(1)}_n=\frac{3^{n-1}(x^{(2)}_1 -x^{(1)}_1)}{4\times \prod_{j=0}^{n-1} 2^{m_j}}
 \end{equation}
 
We define  $ x^{(2)}_1 -x^{(1)}_1$ as follows.
\begin{equation} \label{eqnrhythm}
 x^{(2)}_1 -x^{(1)}_1= R \times 4 \times  \prod_{j=1}^{n-1} 2^{m_j}\doteq R\times D(n)
\end{equation} 
Such a definition enables that $y^{(2)}_n-y^{(1)}_n$ be an integer for any integer $R$. Thus, given any orbit $x^{(1)}_1$, the set of orbits starting with $\{x^{(1)}_1, x^{(1)}_1+ D(n), x^{(1)}_1+ 2D(n),  \cdots\}$ satisfies the requirements of the theorem.	
\end{proof}

Note that, in the above theorem, $R=0$ is corresponding to the trivial case, in which $x^{(2)}_1$ is equal to $x^{(1)}_1$ exactly.   Eqn.~(\ref {eqnrhythm}) can be utilized to look for orbits that share the same orbit rhythm.  For example, given the orbit starting with $x_1=9$ having an orbit rhythm $<m_1=2, m_2=1, m_3=1>$, we can get a class of such orbits starting with $\{9, 41, 73, 105, \cdots \}$ that share the same orbit rhythm,  as $D(3)=4 \times 2^{2+1}=32$.

\section{Discussions}

The Collatz conjecture asserts that every Collatz series always falls into a 4 $\rightarrow$ 2 $\rightarrow$ 1 cycle regardless of $x_1$. If not, there are two counterexamples\cite{Goodwin2015}. One is that there possibly exist non-trivial cycles out of 4 $\rightarrow$ 2 $\rightarrow$ 1 and the other is that there possibly exists $x_1$ for which $\underset{n\rightarrow \infty}{\lim} Col^n(x_1)=\infty $ .

\subsection{Non-trivial cycle}

According to Lemma \ref{lemma:growth}, every growth point is in a format ``4y+3'', so all growth points \emph{must} fall into one of the integer solutions to Eqn.~(\ref{eqn:formula}). The hypothetical existence of a non-trivial cycle indeed implies that there \emph{should} exist a positive integer $k$ for which $ y_{n+k}=y_{n}$ holds, i.e.,
\begin{equation}
	X(x_1, n+k, m_1, m_2, \cdots, m_{n+k-1})=X(x_1, n, m_1, m_2, \cdots, m_{n-1})
\end{equation}
Otherwise, non-trivial cycle \emph{should not} exist, if only there is no integer solution to the above equation.

\subsection{Divergent sequence }
Likewise,  the hypothetical existence of a divergent sequence for $x_1$ implies that there \emph{should} exist infinitely many growth points as $n$ is going to infinity.  Since every growth point is corresponding to a value of $\{y_n\}$, the hypothetical existence of a divergent sequence mandates infinitely many unique values of $\{y_n\}$. That is equivalent to saying, that, for any fixed value $C(x_1)$ that depends on $x_1$, there \emph{should} exist an integer solution to Eqn.\ref{eqn:formula2} for $n> C(x_1)$.  Otherwise,  divergent sequence \emph{should not} exist under the assumption of without non-trivial cycles, if only there is no integer solution to Eqn.\ref{eqn:formula2} for $n> C(x_1)$.

\section{Conclusion}

This paper presents a formula expression, which converts the Collatz conjecture to an equation-solving process. We show that, for any given integer number $n$, there always exists a natural number $x_1$ that starts a monotonic increasing or decreasing sequence of length $n$ with the same step size, if only the corresponding equations have integer solutions. We also derive an equation that can be utilized to look for orbits that share the same orbit rhythm. 

Besides, solving these equations may provide hints for the possible existence of non-trivial cycles and divergent sequences. These issues need further efforts invested.

\section*{Acknowledgment}
This work was supported by the National Natural Science Foundation of China (61273235).

\bibliographystyle{plain}
\bibliography{coll}

\begin{thebibliography}{10}

\bibitem{Aliyev2013}
Y.~N. {Aliyev} and V.~A. {Suleymanov}.
\newblock Construction of periods for 3x+1 problem: Use of division algorithm
  by 2 in 3\mbox{-}{b}ase number system for construction of 3-adic numbers as
  periods of collatz sequence.
\newblock In {\em Proc. 7th Int. Conf. Application of Information and
  Communication Technologies}, pages 1--3, October 2013.

\bibitem{Aljassas2019}
H.~M.~A. {Aljassas} and S.~{Sasi}.
\newblock Performance evaluation of proof-of-work and collatz conjecture
  consensus algorithms.
\newblock In {\em Proc. 2nd Int. Conf. Computer Applications Information
  Security (ICCAIS)}, pages 1--6, May 2019.

\bibitem{Goodwin2015}
J.~R. Goodwin.
\newblock The 3x+1 problem and integer representations.
\newblock {\em arXiv: Number Theory}, 2015.

\bibitem{Korec1994}
I.~Korec.
\newblock A density estimate for the 3x + 1 problem.
\newblock {\em Math. Slovaca}, 44:85--89, 1994.

\bibitem{Krasikov2003}
I.~Krasikov and J.~Lagarias.
\newblock Bounds for the 3x + 1 problem using difference inequalities.
\newblock {\em Acta Arithmetica}, 109:237--258, 2003.

\bibitem{Lagarias1992}
J.~C. Lagarias and A.~Weiss.
\newblock The 3x + 1 problem: two stochastic models.
\newblock {\em Annals of Applied Probability}, 2:229--261, 1992.

\bibitem{Ma2019}
H.~{Ma}, C.~{Jia}, S.~{Li}, W.~{Zheng}, and D.~{Wu}.
\newblock Xmark: Dynamic software watermarking using collatz conjecture.
\newblock {\em IEEE Transactions on Information Forensics and Security},
  14(11):2859--2874, November 2019.

\bibitem{Ren2018}
W.~{Ren}, S.~{Li}, R.~{Xiao}, and W.~{Bi}.
\newblock Collatz conjecture for $2^100000-1$ is true - algorithms for
  verifying extremely large numbers.
\newblock In {\em Proc. Internet of People and Smart City Innovation
  (SmartWorld/SCALCOM/UIC/ATC/CBDCom/IOP/SCI) 2018 IEEE SmartWorld, Ubiquitous
  Intelligence Computing, Advanced Trusted Computing, Scalable Computing
  Communications, Cloud Big Data Computing}, pages 411--416, October 2018.

\bibitem{Sayama2011}
H.~{Sayama}.
\newblock An artificial life view of the collatz problem.
\newblock {\em Artificial Life}, 17(2):137--140, April 2011.

\bibitem{Tao2019}
T.~Tao.
\newblock Almost all orbits of the collatz map attain almost bounded values.
\newblock {\em arXiv:1909.03562v2 [math.PR]}, 2019.

\bibitem{Terras1976}
R.~Terras.
\newblock A stopping time problem on the positive integers.
\newblock {\em Acta Arithmetica}, 30:241--252, 1976.

\end{thebibliography}

\end{document}